\documentclass[a4paper,twoside,reqno]{amsart}
\usepackage{amsmath}
\usepackage{amsfonts}
\usepackage{amsthm}
\usepackage{amssymb}
\usepackage{verbatim}
\usepackage{lingmacros}
\usepackage{hyperref}

\newcommand{\N}{\mathbb N}
\newcommand{\Z}{\mathbb Z}
\newcommand{\R}{\mathbb R}
\newtheorem{theorem}{Theorem}[section]

\newtheorem{problem}[theorem]{Problem}
\newtheorem{example}[theorem]{Example}

\newtheorem{lemma}[theorem]{Lemma}
\theoremstyle{remark}
\newtheorem{remark}[theorem]{Remark}

\begin{document}
\bibliographystyle{plain}
\title{A bound on degrees of primitive elements of toric graph ideals}
\author{Kamil Rychlewicz}
\address{Faculty of Mathematics, Informatics and Mechanics, University of Warsaw, Banacha 2, 02-097 Warsaw, Poland}
\email{kr360316@students.mimuw.edu.pl}
\begin{abstract}
 We prove that for any toric ideal of a graph the degree of any element of Graver basis is bounded above by an exponential function of the maximal degree of a circuit.
\end{abstract}

\maketitle 

\section{Introduction}
Let $a_1,a_2,\dots,a_m\in \Z^n$ be the columns of $n\times m$ matrix $A$ and let $K$ be a field. We consider a homomorphism of $K$-algebras $\phi:K[x_1,x_2,\dots,x_m]\to K[y_1,y_2,\dots,y_n,y_1^{-1},y_2^{-1},\dots,y_n^{-1}]$ defined by $\phi(x_i)=y^{a_i}$, where by definition $y^{(s_1,s_2,\dots,s_n)}=y_1^{s_1} y_2^{s_2}\dots y_n^{s_n}$. Then the \emph{toric ideal} $I_A$ of matrix $A$ is the kernel of $\phi$. We define the \emph{$A$-degree of monomial}:
$$\deg_A\left(x_1^{u_1}x_2^{u_2}\dots x_m^{u_m}\right)=u_1 a_1+u_2 a_2+\dots+u_m a_m\in\Z^n.$$
\begin{theorem}
\label{gene}
For any matrix $A$, $I_A$ is generated by binomials of the form $x^u-x^v$ for which $\deg_A(x^u)=\deg_A(x^v)$.
\end{theorem}
For the proof, see e.g. Lemma 4.1. in \cite{poly}. For a binomial as above, we define $\deg_A(x^u-x^v)=\deg_A(x^u)$.

For any $u\in\ker A\subset \Z^n$ we can consider a binomial $x^{u^+}-x^{u^-}\in I_A$ where $u^+,u^-\in\N^n$ are the unique vectors satysifying $u^+-u^-=u$ and $supp(u^+)\cap supp(u^-)=\emptyset$. In fact, every irreducible binomial $x^{u^+}-x^{u^-}$ has this form. An irreducible binomial $x^{u^+}-x^{u^-}\in I_A$ is called \emph{primitive} if there exists no other binomial $x^{v^+}-x^{v^-}\in I_A$ with $x^{v^+}|x^{u^+}$ and $x^{v^-}|x^{u^-}$. All primitive binomials of $I_A$ constitute its \emph{Graver basis} (see \cite{poly}). An irreducible binomial $x^u-x^v\in I_A$ is called a \emph{circuit} if its support $supp(x^u-x^v)=\left\{i: x_i|x^{u+v}\right\}$ is minimal (with respect to inclusion) among binomials in $I_A$.

Sturmfels conjectured that the degree of any primitive binomial of a toric ideal is bounded above by the maximal degree of a circuit. It was however disproved by Hosten and Thomas (see Example 4.7 in \cite{sturm}). It led to another conjecture (Conjecture 4.8 ibid.): the degree of any primitive binomial is bounded above by the maximal \emph{true degree} of a circuit. The true degree of a circuit $C\in I_A$ is defined as $\deg(C)\cdot index(C)$, where $index(C)$ is the index of the lattice $\Z(a_i: i \in supp(C))$ in $\R(a_i: i\in supp(C))\cap \Z(a_1,a_2,\dots,a_m)$. In \cite{thgrob}
Tatakis and Thoma disproved the conjecture and in \cite{thgrav}
they proved that there is no polynomial bound. They provide counterexamples of toric graph ideals for which the Graver degrees are exponentialy large compared to the true circuit degrees.
A natural question arises: are Graver degrees bounded by \emph{any} function of the maximal true degree of a circuit? The search for counterexamples among graph ideals fails this time and we prove that for toric graph ideals the Graver basis degrees are bounded by an exponential function of maximal true circuit degrees. It's now known that for toric graph ideals the true degree of a circuit is equal to its usual degree (see Theorem 3.1 in \cite{thgrav}), thus we have to bound the degrees of primitive elements by an exponential function of the maximal usual degree of a circuit. This is done in Theorem \ref{final}.

\section{Graver bases and circuits in toric graph ideals}
Let $G$ be a finite simple undirected graph. A \emph{walk} in $G$ is a sequence 
$$(\{v_1,v_2\},\{v_2,v_3\},\dots,\{v_k,v_{k+1}\})$$
of edges of $G$. It is called a \emph{closed walk} if $v_1=v_{k+1}$. If it's closed and $v_1$, $v_2$, \dots, $v_k$ are pairwise distinct, then it is a \emph{cycle} and if $v_1$, $v_2$, \dots, $v_k$, $v_{k+1}$ are pairwise distinct, it's called a \emph{path}. The number of edges in the walk is called its \emph{length} and the walk is \emph{even} (respectively \emph{odd}) if its length is even (respectively odd). An edge (respectively a vertex) of $G$ is called a \emph{cut edge} (respectively a \emph{cut vertex}) if its removal increases the number of connected components of $G$.

Let $a_1,a_2,\dots,a_m\in \Z^n$ be the columns of $n\times m$ matrix $A$. If every $a_j$ for $j=1,2,\dots,m$ is a 0--1 vector and it has exactly two ones, then $A$ is an incidence matrix of some graph $G=(V,E)$ for $V=(v_1,v_2,\dots,v_n)$ and $E=(e_1,e_2,\dots,e_m)$. Then we define the \emph{toric ideal of graph} $G$ as $I_G=I_A\in K[e_1,e_2,\dots,e_m]$. It follows from Theorem \ref{gene} that $I_G$ is generated by elements of the form
$$B_w=\prod_{k=1}^q e_{i_{2k-1}}-\prod_{k=1}^q e_{i_{2k}}$$
for all even closed walks $w=(e_{i_1},e_{i_2},\dots,e_{i_{2q}})$ in graph $G$.
The degree of $B_w$ equals $q$, which is the half of the length of $w$. 

\begin{example}
 If $G$ is a 4-cycle, then its toric ideal is principal generated by element $x_{12}x_{34}-x_{23}x_{14}$ in the polynomial ring $K[x_{12},x_{34},x_{23},x_{14}]$. The variable $x_{ij}$ corresponds to the edge $e_{ij}$ from vertex $i$ to vertex $j$. The walk associated to $x_{12}x_{34}-x_{23}x_{14}$ is the cycle itself -- with $e_{12}$ and $e_{34}$ being its odd edges, $e_{23}$ and $e_{14}$ being its even edges.
\end{example}

In \cite{rees}
and \cite{thmin}
the circuits and primitive elements of toric graph ideals are characterized. We use this results to provide the mentioned exponential bound.

In \cite{rees}
Villareal gave the following description of circuits in toric graph ideal (Proposition 4.2 in the paper):

\begin{theorem}\label{vill}
Let $G$ be a graph. The binomial $B\in I_G$ is a circuit if and only if $B=B_w$ and one of the following holds:
\begin{enumerate}
 \item $w$ is an even cycle;
 \item $w$ consists of two odd cycles with common vertex;
 \item $w$ consists of two odd cycles connected by a path.
\end{enumerate}
\end{theorem}

Note that in the third case, the degree of the circuit equals $\frac{c_1+c_2}2+p$, where $c_1$ and $c_2$ are the sizes of the cycles and $p$ is the length of the path (we'll use it later in the proof of Theorem \ref{final}).

Then in \cite{thmin}
Reyes, Tatakis and Thoma gave a complete description of elements of Graver basis (primitive elements) of the toric ideal of a graph (Corollary 3.2 in the paper):
\begin{theorem}\label{thom}
Let $G$ be a graph. A connected subgraph $W$ of $G$ is an underlying graph of a primitive walk $w$ (i.e. $B_w$ is primitive) if and only if all the following conditions are satisfied:
\begin{enumerate}
 \item every block of $W$ is a cycle or a cut edge;
 \item every cut vertex of $W$ belongs to exactly two blocks and separates the graph in two parts, each of them containing an odd number of edges in cyclic blocks.
\end{enumerate}

\end{theorem}

They also prove that every cut edge of $W$ appears in the $w$ exactly twice and every other edge $W$ appears in $w$ exactly once (Theorem 3.1 in the paper).

\begin{remark}
 Note that if $G$ is not simple, then the cycles can have length 1 (a loop) or 2 (two edges with the same endpoints). Then the above theorems are still true (see Remark 4.17 in \cite{thmin}) as well as the theorem proved below. It should be however noted that in case of a loop, the corresponding entry of the incidence matrix should be equal to 2, not 1.
\end{remark}

Now we prove the following

\begin{theorem}\label{final}
Let $G$ be a graph. Suppose that the degree of every circuit in $I_G$ is bounded above by $n$. Then the degree of any primitive element in $I_G$ is bounded above by $n^2 e^{\frac {2n}e}$.
\end{theorem}

With use of Theorems \ref{vill} and \ref{thom} we give a purely graph-theoretic proof. We begin with a lemma.

\begin{lemma}\label{lem}
 Let $T$ be a tree with at least three vertices. Suppose that for every path $(v_0,v_1,\dots,v_k)$ connecting leaves $v_0$ and $v_k$ in $T$ we have $\sum_{i=1}^{k-1} \deg(v_i)\le M$. Then $T$ has at most $\left(\frac M2+1 \right)e^{\frac Me}$ vertices.
\end{lemma}
\begin{proof}
 Choose any vertex $u$ of $T$ that is not a leaf and make $T$ rooted with $u$ being the root. Let $P$ be the set of all paths $(u=u_0,u_1,\dots,u_s)$ from $u$ to a leaf $u_s$.
 Let $(u_0,u_1,\dots,u_s)\in P$. As $u$ is not a leaf, there exists another path $(u'_0,u'_1,\dots,u'_t)\in P$ such that $u'_1\neq u_1$. Therefore we get a path
 $(u'_t,u'_{t-1},\dots,u'_1,u,u_1,\dots,u_s)$ connecting two leaves. Note that none of the ends of the path is equal to $u$, as $u$ is not a leaf. Thus from the assumption we have
 $$\sum_{i=0}^{s-1} \deg(u_i)+\sum_{i=1}^{t-1} \deg(u'_i)\le M.$$
 Omitting the second term, we get
 \begin{equation}
  \label{simp}
  \sum_{i=0}^{s-1} \deg(u_i)\le M\quad\text{for every}\quad (u_0,u_1,\dots,u_s)\in P.
 \end{equation}
 Then it follows from the arithmetic-geometric mean inequality that
 $$\left(\prod_{i=0}^{s-1}\deg(u_i)\right)^{1/s}\le \frac{\sum_{i=0}^{s-1} \deg(u_i)}s\le \frac Ms.$$
 We get 
 $$\prod_{i=0}^{s-1}\deg(u_i)\le \left(\frac Ms\right)^s=e^{s\ln M-s\ln s}.$$
 By differentiating the exponent (as a function of $s$), we find out that it attains its maximal value $\frac Me$ at $s=\frac Me$. Thus
 \begin{equation}
  \label{exp}
 \prod_{i=0}^{s-1}\deg(u_i)\le e^{\frac Me}.
 \end{equation}
 Now note that
 \begin{equation}
  \label{prod}
\sum_{(u_0,u_1,\dots,u_s)\in P} \frac{1}{\prod_{i=0}^{s-1} \deg(u_i)}\le \sum_{(u_0,u_1,\dots,u_s)\in P} \frac{1}{\deg(u_0)\prod_{i=1}^{s-1}(\deg(u_i)-1)}=1.
 \end{equation}
 The equality above is a known identity --- for a path $(u_0,u_1,\dots,u_s)\in P$ the number 
 $$\frac{1}{\deg(u_0)\prod_{i=1}^{s-1}(\deg(u_i)-1)}$$
 is a probability that going down from the root and choosing the next vertex at random (with uniform probability) at each stage, we end up in $u_s$. Now combining \eqref{exp} and \eqref{prod} we get
 $$|P| \le \sum_{(u_0,u_1,\dots,u_s)\in P} \frac{e^{\frac Me}}{\prod_{i=0}^{s-1}\deg(u_i)}\le e^{\frac Me}.$$
 As $T$ is a tree, every leaf is connected to root by exactly one path, so $|P|$ equals the number of leaves.
 
 From \eqref{simp} we conclude that every leaf has at most $\frac M2$ ancestors, as their degrees are not less than 2 and the sum of their degrees is at most $M$. Every vertex of $T$ that is not a leaf is an ancestor of a leaf (possibly more than one), so there are at most $\frac M2 \cdot |P|$ non-leaves in $T$. It follows that the overall number of vertices in $T$ is at most
 $$\left(\frac M2+1\right)|P|\le \left(\frac M2+1 \right)e^{\frac Me}.$$
\end{proof}

Now we proceed to the proof of the theorem.

\begin{proof}[Proof of Theorem \ref{final}]
 Let $w$ be a primitive walk in the graph $G$ and let $W$ be the underlying subgraph of $w$.
 Consider the block-graph $B(W)$ of $W$ --- a graph whose vertices are the blocks (biconnected components) of $W$ and two blocks are adjacent in $B(W)$ if and only if they share a common vertex (a cut vertex) in $W$. Every block of $B(W)$ is a complete graph of blocks of $W$ sharing a common vertex (see Corollary 1b in \cite{harary}). Thus the second condition from Theorem \ref{thom} implies that all blocks of $B(W)$ are single edges. We conclude that $B(W)$ contains no cycle (because a cycle is biconnected and contains more than one edge). As $w$ is a walk, $W$ is connected and $B(W)$ is connected and therefore it is a tree. For any vertex $v$ of $B(W)$ let $S(v)$ denote its size (i.e. the number of vertices in the corresponding block).
 
 If $B(W)$ has no more than two vertices, $w$ is a circuit (this is an immediate consequence of Theorems \ref{vill} and \ref{thom}) and the conclusion follows. From now on, suppose that $B(W)$ has at least three vertices. The second condition from Theorem \ref{thom} implies that every leaf of $B(W)$ corresponds to on odd cycle in $W$. Therefore for every path $(v_0,v_1,\dots,v_k)$ connecting two leaves of $B(W)$ there exists a circuit which contains two odd cycles ($v_0$ and $v_k$) connected by a path going through $v_1$, $v_2$, \dots, $v_{k-1}$. Moreover, we can construct a circuit whose degree is not less than
 $\frac{\sum_{i=0}^{k} S(v_i)}2,$
 i.e. the length of the path is not less than $\frac{\sum_{i=1}^{k-1} S(v_i)}2$. Indeed, in every block $v_i$ (for $i=1$, $2$, \dots, $k-1$) we have to choose a path from the common vertex of $v_i$ and $v_{i-1}$ to the common vertex of $v_i$ and $v_{i+1}$. As the block is a cycle of size $S(v_i)$ (or an edge for $S(v_i)=2$), we can choose the longer path between those two vertices, whose length is at least $\frac {S(v_i)}2$. On the other hand, from the assumption we know that the degree of the constructed circuit is not greater than $n$. Thus we have
 $$\frac{\sum_{i=0}^{k} S(v_i)}2\le n.$$
 
 Let $\deg(v)$ denote the degree of a block $v$ of $W$ as a vertex of $B(W)$. From Theorem \ref{thom} we know that every cut vertex of a block $v$ is a vertex of exactly one other block, so $v$ has a common point with at most $S(v)$ other blocks, i.e. $\deg(v)\le S(v)$. It follows that
 $$\sum_{i=1}^{k-1} \deg(v_i)=\sum_{i=0}^{k} \deg(v_i)-2\le \sum_{i=0}^{k} S(v_i)-2\le 2n-2$$
 for any path $(v_0,v_1,\dots,v_k)$ as above. Thus we can apply Lemma \ref{lem} for $T=B(W)$ and $M=2n-2$. Then it states that $B(W)$ has at most $ne^{\frac {2n}e}$ vertices.
 
 Now observe that every block in $W$ has at most $2n$ edges (counted with multiplicities in $w$, i.e. every cut edge is counted twice). It's obvious for edges and for even cycles (as they are circuits themselves). If a block is an odd cycle, we can again construct a path from it to another odd cycle (some leaf of $B(W)$) and we get a circuit which by assumption has to have at most $2n$ edges. As there are at most $ne^{\frac{2n}e}$ blocks, the total length of $w$ is at most $2n^2 e^{\frac{2n}e}$ and the degree of $B_w$ is the half of this length, so it's not greater than $n^2e^{\frac{2n}e}$.
 
\end{proof}

\section{Further remarks}
It would be interesting to solve the following
\begin{problem}\label{prob}
 Is the degree of any primitive element of a toric ideal bounded by a function of the maximal true degree of a circuit?
\end{problem}

We provided the positive answer for toric graph ideals only. In \cite{thgrav} toric graph ideals were used as counterexamples to polynomial bounds. It was able because of the pictorial description of their minimal binomials, given in \cite{rees} and \cite{thmin}. Theorem \ref{final} shows that an evidence of possible negative answer to Problem \ref{prob} cannot come from a graph ideal. It would be interesting to solve Problem \ref{prob} for \emph{toric hypergraph ideals} at least. They are exactly the toric ideals associated to 0--1 matrices (see \cite{soncom}, \cite{sonalg}, \cite{bouq} for the strict definition and some results on toric hypergraph ideals). An important step would be to provide a complete characterization of their primitive elements and circuits in spirit of Theorems \ref{vill} and \ref{thom}. This is still an open problem.

\section{Acknowledgements}
I am thankful to Apostolos Thoma for introducing me to toric graph ideals, fruitful discussion and pointing out the mistakes.

\end{document}